\documentclass{amsart}

\usepackage{amsmath,amsthm,amssymb}

\DeclareMathOperator{\degree}{\text{deg}}
\DeclareMathOperator{\flambda}{\mathcal{F}_\lambda}

\DeclareMathOperator{\A}{\mathfrak{A}}
\DeclareMathOperator{\gl}{\mathfrak{gl}}

\DeclareMathOperator {\C}{\mathbb{C}}

\DeclareMathOperator{\uu}{\mathcal{U}}
\DeclareMathOperator{\ff}{\mathcal{F}}
\newcommand{\bonebar}{\ensuremath{\overline{B}_1}}

\newcommand{\cAlg}{\cC\text{-Alg}}

\newcommand{\Vect}{\operatorname{Vect}}

\newcommand{\End}{\operatorname{End}}
\newcommand{\Hom}{\operatorname{Hom}}

\newcommand{\sgn}{\text{sgn}}

\newcommand{\cC}{\mathcal{C}}

\newcommand{\ot}{\otimes}
\newcommand{\bt}{\boxtimes}
\newcommand{\id}{\operatorname{id}}

\newcommand{\CC}{\mathbb{C}}
\newcommand{\NN}{\mathbb{N}}
\newcommand{\ZZ}{\mathbb{Z}}

\def\HH{\hbox{${\mathcal H}$\kern-5.2pt${\mathcal H}$}}

\newtheorem{theorem}{Theorem}[section]
\newtheorem{lemma}[theorem]{Lemma}
\newtheorem{conjecture}[theorem]{Conjecture}
\newtheorem{claim}[theorem]{Claim}
\newtheorem{proposition}[theorem]{Proposition}
\newtheorem{corollary}[theorem]{Corollary}
\theoremstyle{definition}
\newtheorem{definition}[theorem]{Definition}
\newtheorem{example}[theorem]{Example}
\newtheorem{remark}[theorem]{Remark}

\begin{document}

\author{Asilata Bapat}
\email{asilata@mit.edu}
\author{David Jordan}
\email{djordan@math.mit.edu}
\title{Lower central series of free algebras in symmetric tensor categories}


\begin{abstract} We continue the study of the lower central series of a free associative algebra, initiated by B. Feigin and B. Shoikhet \cite{FS}.  We generalize via Schur functors the constructions of the lower central series to any symmetric tensor category; specifically we compute the modified first quotient $\overline{B}_1,$ and second and third quotients $B_2,$ and $B_3$ of the series for a free algebra $T(V)$ in any symmetric tensor category, generalizing the main results of \cite{FS} and \cite{AJ}.  In the case $A_{m|n}:=T(\CC^{m|n})$, we use these results to compute the explicit Hilbert series.  Finally, we prove a result relating the lower central series to the corresponding filtration by two-sided associative ideals, confirming a conjecture from \cite{EKM}, and another one from \cite{AJ}, as corollaries.
\end{abstract}
\maketitle
\section{Introduction}
The \emph{lower central series} of an associative algebra $A$ is the descending filtration by Lie ideals $A=L_1(A)\supset L_2(A)\supset\cdots$ defined inductively by $L_{k+1} := [A,L_k]$ for $k\geq 2$. The corresponding associated graded Lie algebra is denoted by $B(A) = \bigoplus_{i=0}^\infty B_k$, where $B_k:= L_k/L_{k+1}$.  We let $M_k:=AL_k$ denote the two-sided associative ideal generated by $L_k$, and let $N_k:=M_k/M_{k+1}$.  We let $Z$ denote the image of $M_3$ in $B_1$, and $\overline{B}_1:=B_1/Z$.

The notions of associative algebras and Lie algebras make sense in any symmetric tensor category $\cC$, as does the lower central series filtration.  These constructions are given in Section 3; in particular, we consider the lower central series of the tensor algebra $T(V)$ of any object $V\in\cC$.  In this generality, the lower central series quotients $B_k$ are functorial in $V$, and may be expressed in the basis of Schur functors $\mathbb{S}_\lambda$, as is explained in Section 4.

It follows that there exists a universal decomposition of the functor $V\mapsto B_k(T(V))$ into a sum of Schur functors $\mathbb{S}_\lambda$, which holds in any symmetric tensor category.  Moreover, to compute this decomposition it suffices to consider only the cases $\cC=\Vect, V=\CC^n$, as $n\to\infty$. As an application, we have:

\begin{theorem} \label{B1B2B3}For any symmetric tensor category $\cC$, and any $V\in\cC$, the maps $\phi$ from Theorem \ref{thm:fsisomorphism}, and $f_3$ from Theorem \ref{thm:AJisomorphism}, induce natural isomorphisms:
\begin{align*}\overline{B}_1(T(V)) &\cong \Omega^{ev}(V)/\Omega^{ev}_{ex}(V),\\
B_2(T(V)) &\cong \Omega^{ev}_{ex}(V),\\
B_3(T(V)) &\cong S(V)\ot (\oplus_{k=0}^{\infty} \mathbb{S}_{(2,1^{2k+1})}(V)).
\end{align*}
\end{theorem}

In Section 5, we consider the free algebra $A=A_{m|n}=T(\CC^{m|n})$ on even generators $x_1,\ldots x_m$ and odd generators $y_1,\ldots y_n$; we study the lower central series with respect to the usual super commutator $[a,b]:=ab - (-1)^{|a||b|}ba$.  The algebra $A$ and thus each of the $B_k$ are graded by the degree in the generators.  Let us write $h_{B_k}(\mathbf{u},\mathbf{v})$ for the multi-graded Hilbert series of the $B_k$, where $\mathbf{u}=(u_1,\ldots,u_m), \mathbf{v}=(v_1,\ldots,v_n)$ record the degree in the even and odd variables, respectively.
Our first main result is a computation of the Hilbert series for $\overline{B}_1$, $B_2$ and $B_3$, based on the above decomposition:

\begin{theorem} \label{B1B2B3HS} We have the following Hilbert series:
\begin{align*}h_{\overline{B}_1}(\mathbf{u},\mathbf{v})&= \frac{1}{4}(\mathbf{P}_{m|n} + \sum_{i=1}^m\frac{u_i}{2(1-u_i)} + \sum_{j=1}^n\frac{v_j}{2(1+v_j)} + 3),
 \\
h_{B_2}(\mathbf{u},\mathbf{v})&= \frac{1}{4}(\mathbf{P}_{m|n} - \sum_{i=1}^m\frac{u_i}{2(1-u_i)} - \sum_{j=1}^n\frac{v_j}{2(1+v_j)} - 1),\\
h_{B_3}(\mathbf{u},\mathbf{v})&=\frac12((\sum_i u_i + \sum_j v_j)(\mathbf{P}_{m|n} +1) - (\mathbf{P}_{m|n} -1)),
\end{align*}
where $\mathbf{P}_{m|n}:=\prod_{i=1}^m\frac{(1+u_i)}{(1-u_i)}\cdot\prod_{j=1}^n\frac{(1+v_j)}{(1-v_j)}$,
\end{theorem}

Our second main result, presented in Section 6, relates the Lie ideals $L_k$ for any algebra $A$ to the two-sided associative ideals $M_k:=A L_k$.  More precisely, we prove the following theorem (for any algebra in any symmetric tensor category $\cC$):
\begin{theorem}\label{mainident} $[M_j,L_k]\subset L_{k+j}$, whenever $j$ is odd. \end{theorem}
In the case $A=A_n$, this may be seen as a strengthening of a special case of ``Key-Lemma" of \cite{FS}, which asserts that $[M_3,L_1]\subset L_3$.
While the proof of this theorem is largely by direct computation, there are many implications.  Firstly, we prove the ``if" direction of \cite{EKM}, Conjecture 3.6 as a corollary:
\begin{corollary}\label{mjmkcor} We have $M_jM_k\subset M_{j+k-1}$, whenever $j$ or $k$ is odd.\end{corollary}

As another application of Theorem \ref{mainident}, we have:
\begin{corollary}\label{Bmcor} $B_k= [A_{\leq 2},B_{k-1}],$ for all $k$. \end{corollary}
A central observation of \cite{FS} was that the lower central series quotients $B_k, k\geq 2$ each carry an action of the Lie algebra $W_n$ of vector fields on $\CC^n$ and that they are iterated extensions of so-called ``tensor field modules'' $\flambda$ associated to a Young diagram $\lambda$.   Theorem \ref{Bmcor} was conjectured in \cite{AJ}, Remark 1.4, where it was explained (see also the proof of Lemma 5.4, {\it loc. cit.}) how to derive the following:
\begin{corollary}
  \label{thm:strongerlambdabound}
  Let $k\geq 3$. For $\flambda$ in the Jordan-H\"older series of $B_k(A_n)$, we have:
    \[
    |\lambda|\leq 2k-3+2\left\lfloor \frac{n-2}{2}\right\rfloor.
    \]
\end{corollary}

Let $\overline{\lambda}$ denote the Young diagram obtained by deleting the first column of $\lambda$.

\begin{corollary}
  \label{cor:lambdabound}
Let $k\geq 3$. For $\flambda$ in the Jordan-H\"older series of $B_k(A_n)$, we have:
    \[
    |\overline{\lambda}|\leq 2k-5.
    \]
\end{corollary}

Let $\mathcal{G}_\lambda(V):=\mathbb{S}_\lambda(V)\ot \Omega^{ev}(V)$.   For example, $A/M_3\cong \mathcal{G}_{(0)}$, and $B_3\cong \mathcal{G}_{(2,1)}$, by Theorem \ref{B1B2B3}.  As another consequence of Theorem \ref{Bmcor}, we have:

\begin{proposition}\label{glambdabound} Let $\cC$ be a symmetric tensor category, let $V\in\cC$, and let $A=T(V)$.  There exists a finite collection $\Lambda_k$ of Young diagrams $\lambda$, and a surjection $\oplus_{\lambda\in\Lambda_k}\mathcal{G}_{\lambda} \to B_k$.
\end{proposition}

For a graded vector space $M$, let $M[d]$ denote the $d$th graded component.  We have:

\begin{corollary}\label{weakbound}Let $k\geq 3$, and let $A:=A_{m|n}$.  There exists $C_k(m,n)$ such that $\dim B_k(A_{m|n})[d]\leq C_k(m,n)d^{m+n-1}$.
\end{corollary}

In Section 8, we state some open questions and conjectures, and suggest possible approaches to their solution.

\subsection*{Acknowledgments} We are deeply grateful to P. Etingof for his guidance and advice.  The work of the first author was supported by MIT's SPUR and UROP programs for undergraduate research.  The work of the second author was supported by NSF grant DMS-0504847.  We are also grateful to the referee and to Alexei Krasilnikov, for catching an error in a previous proof of Lemma \ref{keylemma}.

\section{Preliminaries}
In this section, we recall definitions for the lower central series of an associative algebra $A$, its associated graded Lie algebra $B(A)$, the Lie algebras $W_n$, and the tensor field modules $\flambda$.
\begin{definition}
Let $W_n:=Der(\C[x_1,\ldots,x_n])$ denote the Lie algebra of polynomial vector fields on $\CC^n$. Then $W_n\cong\bigoplus_i\C[x_1,\ldots,x_n]\partial_i$,
with bracket 
\[
[p\partial_i,q\partial_j] = p\frac{\partial q}{\partial x_i}\partial_j - q\frac{\partial p}{\partial x_j}\partial_i.
\]
\end{definition}
Let $W_n^0$ denote the Lie subalgebra of vector fields vanishing at the origin, and let $W_n^{00}\subset W_n^0$ denote its Lie ideal of vector fields vanishing at the origin to at least  second order. Then $W_n^0/W_n^{00}$ is isomorphic to $\gl_n$, via the map $x_i\partial_j\mapsto E_{ij}$. Let $\lambda = (\lambda_1\geq\cdots\geq\lambda_n)$ be a Young diagram, and let $V_\lambda$ be the corresponding irreducible $\gl_n$-module. Then $V_\lambda$ is also a representation of $W_n^0$ via the projection $W_n^0\to W_n^0/W_n^{00}\cong \gl_n$. 

Let $\widetilde{\ff}_\lambda=  \Hom^{fin.}_{\uu(W_n^0)}(\uu(W_n),V_\lambda)$, denote the finite part of the coinduced module $\Hom_{\uu(W_n^0)}(\uu(W_n),V_\lambda)$, spanned by homogeneous vectors. As a vector space, $\widetilde{\ff}_\lambda$ is isomorphic to $\C[x_1,\ldots,x_n]\otimes V_\lambda$.
\begin{theorem}[\cite{ANR}]\label{flambdairrep}
If $\lambda_1\geq 2$ or if $\lambda = (1^n)$, then $\widetilde{\ff}_\lambda$ is an irreducible $W_n$-module. Otherwise, $\lambda = (1^k,0^{n-k})$, and $\widetilde{\ff}_\lambda$ is $\Omega^k(\C^n)$, the space of polynomial differential $k$-forms on $\C^n$. In this case $\widetilde{\ff}_\lambda$ contains a unique irreducible submodule consisting of the closed $k$-forms.
\end{theorem}

Let $\flambda$ denote the unique irreducible submodule of $\widetilde{\ff}_\lambda$, which is equal to $\widetilde{\ff}_\lambda$ if $\lambda_1\geq 2$ or $\lambda = (1^n)$.

\begin{theorem}[\cite{ANR}]
Any $W_n$-module on which the operators $x_i\partial_i$, for $i = 1,\ldots , n$, act semisimply with nonnegative integer eigenvalues and finite dimensional common eigenspaces has a Jordan-H\"older series whose composition factors are $\flambda$, each occurring with finite multiplicity.
\label{thm:rudakovwn}
\end{theorem}

In \cite{FS}, it was realized that for $A=A_n$, each $B_k$, $k\geq 2$ carries a $W_n$-action, arising from an identification of $A/M_3$ with the algebra $\Omega^{ev}$ of even degree differential forms on $\CC^n$.  Let $\Omega^{ev}_{ex}$ denote the subspace of even degree exact forms. Let us denote by ``$\wedge$'' the wedge product of differential forms.  We let ``$\ast$'' denote the \emph{Fedosov product},
$$a\ast b := a\wedge b + (-1)^{\degree a}da\wedge db,$$
which defines an associative multiplication on $\Omega,$ preserving $\Omega^{ev}$.  The space $\Omega$ (resp. $\Omega^{ev}$) equipped with the Fedosov product is denoted by $\Omega_\ast$ (resp. $\Omega^{ev}_\ast$). We have:

\begin{theorem}[\cite{FS}]
The map $\phi: x_i\mapsto x_i$ extends to an isomorphism of algebras $\Omega^{ev}_\ast\cong A/M_3(A)$, restricts to an isomorphism $\Omega^{ev}_{ex}\cong B_2(A)$, and descends to an isomorphism of  $\Omega^{ev}/\Omega^{ev}_{ex}\cong \bonebar(A)$.\label{thm:fsisomorphism}
\end{theorem}

\begin{theorem}[\cite{FS}]
The action of $W_n$ on $\bonebar\cong\Omega^{ev}/\Omega^{ev}_{ex}$ by Lie derivatives extends uniquely to an action of $W_n$ on $B_k$, for $k\geq 2$. 
\end{theorem}

This action of $W_n$ clearly satisfies the conditions of Theorem \ref{thm:rudakovwn}.  It follows that $\overline{B}_1$ and each $B_k, k\geq 2$ have a Jordan-H\"older series with respect to this action, such that each composition factor is isomorphic to $\flambda$ for some Young diagram $\lambda$. By \cite{DE}, Corollary 3, the Jordan-H\"older series is of finite length.
In \cite{AJ}, similar techniques were used to give a description of $B_3(A_n)$.  We have:

\begin{theorem}\label{thm:AJisomorphism}
The map $f_3:\overline{B}_1\otimes B_2\to B_3, a\ot b\mapsto [a,b]$ restricts to a surjection $(\Omega^0/\CC)\otimes B_2\to B_3$, and induces an isomorphism
$$B_3\cong \oplus_{k=0}^{\infty} \mathcal{F}_{(2,1^{2k+1})}.$$
\end{theorem}

\section{Lower series filtrations in symmetric tensor categories}
Let $(\cC,\sigma)$ be a symmetric tensor category.  For the sake of clarity, we suppress explicit mention of associators in all formulas and commutative diagrams in this section.
\begin{definition} A (unital) associative algebra $(A,m,\eta)$ in $\cC$ is an object $A\in\cC$, together with morphisms $\eta:\mathbf{1}\to A$ and $m:A\ot A\to A$ such that $m\circ(\eta\ot\id)=m\circ(\id\ot\eta)=\id$ and $m\circ(m\ot\id)=m\circ(\id\ot m)$.  Algebra morphisms are defined in the usual way, yielding the category $\cAlg$.  We will say that $A\in\cAlg$ is commutative if $m\circ(\id_{A\ot A}-\sigma_{A,A})=0$.

\end{definition}

\begin{definition}A Lie algebra $(\A,\mu)$ in $\cC$ is an object $\A$ and a morphism $\mu:\A\ot\A\to\A$ satisfying:
\begin{enumerate}
\item Skew symmetry:  $\mu\circ(\id_{\A\ot\A} + \sigma_{\A,\A})=0$.
\item Jacobi identity:  $\mu\circ(\id_{\A}\ot\mu)\circ(\id + (123) + (132))=0$.
\end{enumerate}
Lie algebra morphisms are defined in the usual way, yielding the category $\cC$-Lie.
\end{definition}
As with the example $\cC=\Vect$, we have the functor,
 $$\operatorname{Lie}:\cAlg\to\cC\text{-Lie}$$
$$(A,m,\eta)\mapsto (A,m\circ(\id - \sigma_{A,A})).$$ 

\begin{definition} The \emph{commutator morphisms} $l_k:\A^{\ot k}\to \A$ are defined inductively as follows: $l_1:=\id_{\A}, l_2:=\mu, l_m:=l_{k-1}\circ (\id^{k-2}\ot l_2)$.  The \emph{lower central series filtration} is the collection of $L_k(\A):=\operatorname{Im}(l_k)\subset \A$.  For each $k\in \NN$, we have injections $i_k: L_{k+1}\to L_k$ induced by the natural injections $i_1:L_2\to \A$.  We let $B_k:=\operatorname{coker} i_k$ denote the associated graded components.
\end{definition}

The following is a straightforward generalization of the corresponding statement for $\cC=\Vect$, which is proved by repeated use of Jacobi identity:
\begin{proposition}
Let $f:\A^{\ot k}\to \A$ be any morphism obtained by iterating $\mu$.  Then $\operatorname{Im}(f)\subset L_k$.\end{proposition}

\begin{example}
We will consider the following examples of algebras and their associated Lie algebras.  Let $V\in\cC$.
\begin{enumerate}
\item The tensor algebra $T(V):=\oplus_i V^{\ot i}$ with the usual product.
\item The symmetric algebra $S(V):=T(V)/\langle\operatorname{Im}(\sigma_{V,V}-\id_{V\ot V})\rangle$.
\item The exterior algebra $\Lambda(V):=T(V)/\langle\operatorname{Im}(\sigma_{V,V} + \id_{V\ot V})\rangle$.
\item The algebra of differential forms $\Omega(V):=S(V)\ot\Lambda(V)$, and its commutative subalgebra of even forms (relative to grading on $\Lambda(V)$).
\end{enumerate}
\end{example}

Recall that we have canonical isomorphisms $s:V\to S^1(V)$ and $\lambda:V\to \Lambda^1(V)$.  The algebra $\Omega(V)$ has a unique derivation $d:\Omega^*(V)\to\Omega^{* + 1}(V)$, such that $d|_{S^1(V)}=\lambda\circ s^{-1}$, and $d|_{\Lambda^1(V)}=0$.  As with ordinary differential forms, $d$ defines a differential: $d^2=0$.  We let $\Omega(V)_{ex}=\operatorname{Im}(d)$ denote the exact forms, and $\Omega(V)_{cl}=\ker(d)$ the closed forms.  By the Poincar\'e lemma, we have an isomorphism $\Omega(V)_{cl}\cong \mathbf{1} \oplus \Omega(V)_{ex}$.

\section{From $\text{Vect}$ to any symmetric tensor category}
All definitions in this section are over a field $k$ of characteristic zero.  For background on symmetric tensor categories and Schur functors, see \cite{D} and \cite{EH}.
\begin{definition} The abelian category of abstract Schur functors is $\mathbf{Sch}:=\oplus_{j\geq 0} \textrm{Rep}(S_j)$.  The tensor product of $V\in \textrm{Rep}(S_j)$ and $W\in\textrm{Rep}(S_k)$ is:
$$V\ot W:=\text{Ind}_{S_j\times S_k}^{S_{j+k}} V\bt W,$$
which defines the structure of a tensor category on $\mathbf{Sch}$.  The isomorphisms $S_j\times S_k\cong S_k\times S_j$ induced by $(1,\ldots, j, j+1,\ldots j+k)\mapsto (j+1,\ldots j+k,1,\ldots j)$ make $\mathbf{Sch}$ a symmetric tensor category.
\end{definition}

As an abelian category, $\mathbf{Sch}$ is filtered by its full subcategories $\mathbf{Sch}_N:=\oplus_{j\leq N}\textrm{Rep}(S_j)$, and is semi-simple, with simple objects $W_\lambda$, where $W_\lambda\in\textrm{Irrep}(S_k)$ is the irreducible indexed by the Young diagram of size $k=|\lambda|$, and $k$ ranges over $\ZZ_{\geq 0}$.

Let $\cC$ be a symmetric tensor category.  We have a bi-functor $F:\mathbf{Sch}\times \cC\to\cC$, sending $(W\in \textrm{Rep}(S_k),V\in\cC)$ to $F_{V}(W):=(W\ot V^{\ot k})^{S_k}\in\cC$.  More precisely, $F_V(W)$ is the image of $e_k=\frac{1}{k!}\sum_{\sigma\in S_k}\sigma\otimes \sigma\in \End_\cC(W\ot V^{\ot k})$.  For a Young diagram $\lambda$ of size $k$, we let $\mathbb{S}_\lambda:=V\mapsto F_{V}(W_\lambda)$, called the irreducible Schur functor in $\cC$ of type $\lambda$.  Then $F_V:\mathbf{Sch}\to\cC$ is a symmetric tensor functor.

For the symmetric tensor category $\textrm{Rep}(GL_n)$ of locally finite $GL_n$-modules, we consider the functors $F_n:=F_{\CC^n}:\mathbf{Sch}\to \textrm{Rep}(GL_n)$.

\begin{claim}\label{fullyfaithful} The family of functors $F_n$ is asymptotically fully faithful.  More precisely, $F_n|_{\mathbf{Sch}_N}$ is fully faithful for $n\geq N$.
\end{claim}
\begin{proof}
Let $V=\CC^n$.  On simples $W_\lambda\in \textrm{Rep}(S_N)$, $F_n(W_\lambda)$ is simply the $W_{\lambda}$-isotypic component of the $S_N$ module $V^{\ot N}$.  By Schur-Weyl duality, this is an irreducible $GL_n$-module $V_{\lambda}$ for $n\geq N$ (while for $N> n$, it may be zero).
We have $$Hom_{GL_n}(V_{\lambda},V_{\mu}) \cong \CC^{\delta_{\lambda,\mu}}\cong \Hom_{S_N}(W_\lambda,W_\mu).$$  
\end{proof}

We will apply this claim in the form of the following:

\begin{proposition}  \label{meta}
\begin{enumerate}\item Suppose we have an identity of Schur functors (up to an isomorphism) that holds in $\textrm{Rep}(GL_n)$, for all $n$.  Then it holds for abstract Schur functors, and hence in any symmetric tensor category.
\item Any morphism of abstract Schur functors which yields an isomorphism (epi, mono, or zero morphism, resp.) in $\textrm{Rep}(GL_n)$ for all $n$ is itself an isomorphism (epi, mono, or zero morphism, resp.), and hence yields an isomorphism (epi, mono, or zero morphism, resp.) in any symmetric tensor category.
\end{enumerate}
\end{proposition}

A simple, yet important, observation is that the members $L_k:=L_k(T(V))$ of the lower central series, and the quotients $B_k$, as well as all the algebras and Lie algebras constructed in the previous sections, are defined purely in terms of the symmetric group action on various tensor products of an object $V$.  Thus each of these constructions is in fact functorial in $V$, and moreover can be expressed in the basis of Schur functors $\mathbb{S}_\lambda$.

An immediate consequence is that if we can provide a consistent decomposition of $L_i(T(\CC^n))$ into a sum of $\mathbb{S}_\lambda(\CC^n)$, for all $n$, then that formula holds in any symmetric tensor category.  As an application, we have:

\begin{proof}[Proof of Theorem \ref{B1B2B3}]  Both sides of each asserted isomorphism are expressed in terms of the Schur functors $\mathbb{S}_\lambda(V)$.  It is easily checked that the morphisms from \cite{FS} and \cite{AJ} are natural in $V$, and defined in any symmetric tensor category; thus they define a natural transformation of the corresponding functors.  Moreover it is shown that they are natural isomorphisms in the case $\cC=\textrm{Rep}(GL_n)$, for all $n$, and therefore, by Proposition \ref{meta}, in any tensor category.
\end{proof}

\section{Super vector spaces and $A_{m|n}$}
Now we apply the definitions and propositions of the previous section to the setting of super vector spaces.  Let $\cC:=\text{SuperVect}$, and let $\CC^{m|n}\in\cC$ denote the super vector space on even basis $\{x_1,\ldots,x_m\}$ and odd basis $\{y_1,\ldots,y_n\}$. Let $A:=A_{m|n}$ denote the tensor algebra $T(\CC^{m|n})$, and consider the lower central series $L_k(A)$, and the associated graded quotients $B_k(A)$. 

\begin{proof}[Proof of Theorem \ref{B1B2B3HS}]
We compute $h_{B_2}$ first.  By Theorem \ref{B1B2B3}, we have that $B_2$ is the subspace $\Omega^{ev}_{ex}$ of exact super differential forms of even degree:
$$h_{B_2} = h_{\Omega^{ev}_{ex}} = \sum_{i=0}^\infty h_{\Omega^{2i}_{ex}}.$$

The map $d\colon\Omega(\C^{m|n})\to\Omega(\C^{m|n})$ gives rise to the following complex:
$$
\cdots\stackrel{d}{\longrightarrow}\Omega^{i-1}\stackrel{d}{\longrightarrow}\Omega^{i}\stackrel{d}{\longrightarrow}\Omega^{i+1}\stackrel{d}{\longrightarrow}\cdots .
$$
This sequence is exact except for a copy of $\C$ at $i=0$. 
Therefore we obtain the recursion relation:
\begin{align}
  \label{eq:1}
  h_{\Omega^{i}_{ex}} &= h_{\Omega^{i-1}} - h_{\Omega^{i-1}_{ex}}, \,\,(i\geq 2) \\
  h_{\Omega^1_{ex}} &= h_{\Omega^{0}} - 1,\nonumber
\end{align}
which has solution:
\begin{equation}
  \label{eq:2}
  h_{\Omega^{i}_{ex}} = \sum_{j=0}^{i-1}(-1)^{i-j+1}h_{\Omega^j} + (-1)^{i+1}.
\end{equation}
By the tensor decomposition, $\Omega(\C^{m|n}) = S(\C^{m|n})\otimes\Lambda(\C^{m|n})$, we may write:
$$h_\Omega(\mathbf{u},\mathbf{v},t) = \frac{\prod_{j=1}^n(1+v_j)}{\prod_{i=1}^m(1-u_i)}\cdot \frac{\prod_{i=1}^m(1+tu_i)}{\prod_{j=1}^n(1-tv_j)},$$
\noindent where the variable $t$ is a counter for the degree of the form. Thus $h(\Omega^i)$ is the coefficient of $t^i$ in $h(\Omega)$. From equations \eqref{eq:1} and \eqref{eq:2}, we have:
\begin{align*}
h_{\Omega^{2k}_{ex}} &= \sum_{i=0}^{2k-1}(-1)^{i+1}\text{Coeff}_{t^i}(h_\Omega) +1,\\
& = \text{Res}_{t=0}\left(h_\Omega\cdot\frac{1}{t}\cdot\frac{t^{-2k}-1}{1+t^{-1}}\right) +1.
\end{align*}
As the term $(t^{-2k}-1)/(1+t^{-1})$ has a convergent power series for $|t|>1$, we can compute this residue by taking a contour integral around a circle $\gamma$ with centre at the origin and radius greater than $1$: 
$$
h_{\Omega^{2k}_{ex}} =\frac{1}{2\pi i}\int_{\gamma}\left(h_\Omega\cdot\frac{-1}{1+t}\right) +\frac{1}{2\pi i}\int_{\gamma}\left(h_\Omega\cdot\frac{t^{-2k}}{1+t}\right) +1.
$$

An elementary computation shows that
$$
\frac{1}{2\pi i}\int_{\gamma}\left(h_\Omega\cdot\frac{-1}{1+t}\right) = -1,
$$
so we find
$$
  h_{\Omega^{2k}_{ex}} = \frac{1}{2\pi     i}\int_{\gamma}\left(h_\Omega\cdot\frac{t^{-2k}}{1+t}\right),
$$
and thus:
$$
h_{\Omega^{ev}_{ex}} =\sum_{k\geq 1}h_{\Omega^{2k}_{ex}}= \frac{1}{2\pi i}\int_{\gamma}\left(h_\Omega\cdot\frac{1}{(1+t)(1-t^{-2})}\right) -1.
$$

The integrand has a simple pole at $t=1$ and a double pole at $t=-1$. Calculating the integral by the method of residues yields us the formula asserted for $h_{B_2}$.

We can compute the Hilbert series of $\Omega^{ev}$ directly:
\begin{align*}h_{\Omega^{ev}} &= \frac12(h_{\Omega}(\mathbf{u},\mathbf{v},1) + h_{\Omega}(\mathbf{u},\mathbf{v},-1))\\
&=\frac12\cdot\frac{\prod_{j=1}^n(1+v_j)}{\prod_{i=1}^m(1-u_i)}\cdot 
(\frac{\prod_{i=1}^m(1+u_i)}{\prod_{j=1}^n(1-v_j)} + \frac{\prod_{i=1}^m(1-u_i)}{\prod_{j=1}^n(1+v_j)})\\
&=\frac12\cdot( \mathbf{P}_{m|n} + 1).
\end{align*}
Thus, we compute:
\begin{align*} 
h_{\overline{B}_1}= h_{\Omega^{ev}} - h_{B_2} = \frac{1}{4}(\mathbf{P}_{m|n} + \sum_{i=1}^m\frac{u_i}{2(1-u_i)} + \sum_{j=1}^n\frac{v_j}{2(1+v_j)} + 3),
\end{align*}
as asserted.

Finally, we can compute the Hilbert series of $B_3$ from its tensor decomposition in Theorem \ref{B1B2B3}.  We have:
$$h_{\mathbb{S}_{(2,1^k)}(V)} = h_V\cdot h_{\mathbb{S}_{(1^{k+1})}(V)} - h_{\mathbb{S}_{(1^{k+2})}(V)},$$
from which we compute:
\begin{align*}
h_{B_3}&= h_{S(V)}\left(h_V \sum_{k=1}^\infty h_{\mathbb{S}_{(1^{2k})}} - \sum_{k=1}^\infty h_{\mathbb{S}_{(1^{2k+1})}}\right)\\
&= h_V(h_{\Omega^{ev}}-1) - (h_{\Omega^{odd}}-h_V)\\
&= \frac12((\sum_i u_i + \sum_j v_j)(\mathbf{P}_{m|n} +1) - (\mathbf{P}_{m|n} -1)).
\end{align*}
\end{proof}

\section{Relations between $M_j$ and $L_k$}
All statements in this section, as well as Lemma \ref{B4cor} and Theorem \ref{Bmcor} apply to any symmetric tensor category, while proofs are carried out only in the category of vector spaces.  The proofs in general follow by application of Proposition \ref{meta}.
\subsection*{Notational conventions}
We have to write many expressions involving iterated commutators:  when it is clearer, we omit brackets and adopt the convention of right-iterated bracketing:  $$[x_1,\ldots,x_n] := [x_1,[x_2,[\cdots,[x_{n-1},x_n]\cdots].$$  We let ``$\star$'' denote the symmetric product:  $a\star b = \frac12(ab + ba)$.  Let $Sym(X)$ denote the symmetric group on the set $X$, and let
$$G:=Sym(\{x,y,v\})\times Sym(\{u,z\})\subset Sym(\{x,y,z,u,v\}), \textrm{ and}$$
$$\textrm{Alt}_G = \sum_{g\in G} \sgn(g) g\in \CC[G].$$

\begin{proof}[Proof of Theorem \ref{mainident}]
The proof is a double induction, first on $k$, then on $j$.  The base of induction is $(j,k)=(3,1)$, as the trivial case $(1,1)$ is not sufficient for the induction step.
\begin{lemma}\label{keylemma}$[M_3,L_1] \subset L_4$.\end{lemma}
\begin{proof} 
We show that $[x[y,z,u],v]\in L_4$;  letting $x,y,z,u,v$ range over all monomials, we span $[M_3,L_1]$, proving the lemma.  We may replace $[x[y,z,u],v]$ by $[x\star[y,z,u],v]$, as they are congruent modulo $L_5$.  The following two identities are straightforward:
$$ [x\star[y,z,u],v]+[y\star[x,z,u],v] = [[x\star y,z,u],v] \in L_4$$
$$[x\star[y,z,u],v]+[v\star[y,z,u],x]=-[x\star v,y,z,u]\in L_4.$$
Together they imply that $[x\star[y,z,u],v]$ is alternating for $G$, modulo $L_4$.  Finally, we have the following identity, which may be verified directly on coefficients of the ${5\choose 2}=10$ monomials where $x$ is left of $y$ is left of $v$ and $z$ is left of $u$:
$$\textrm{Alt}_G [x\star[y,z,u],v] = \textrm{Alt}_G(4[z\star x,y,v,u] - 2[x,z,y,u\star v]).$$
We conclude that $[x[y,z,u],v]\in L_4$.  
\end{proof}
\begin{lemma} $[M_3,L_k]\subset L_{k+3}$.\end{lemma}
\begin{proof} Consider $[a[b,c,d],l]$, with $l \in L_k$.  Write $l=[x,y]$, where $x\in L_1, y\in L_{k-1}$.  We have $$[a[b,c,d],[x,y]] = [[a[b,c,d],x],y]+ [x,[a[b,c,d],y]]\subset L_{k+3},$$  by induction.
\end{proof}
To conclude the proof of the theorem, we consider a general element $[m,l]$ with $m\in M_j$, and $l=[l_1,\ldots,l_k]\in L_k$. Write $m= a[b,c,d]$, where $d\in L_{j-2}$.  Then by Lemma \ref{keylemma}, we have that $[a[b,c,d],l] \subset L_{k+3}(a,b,c,d,l_1,\ldots l_k)$, meaning that we regard $d$ as a variable instead of an iterated bracket.  Each summand in $L_{k+3}$ is an iterated bracket of linear and quadratic expressions in the variables $a,b,c,d,l_1,\ldots l_k$.  By repeated application of the Jacobi identity, we may assume that all summands are of the form $[y_1,\ldots,y_{k+2},y_{k+3}d],$ for various permutations $\{y_1,\ldots, y_{k+3}\}=\{a,b,c,l_1,\ldots,l_k\}$.  We now plug in $d=[d_1,\ldots, d_{j-2}]$.  Then $[y_{k+2},y_{k+3}d]\subset L_{j-1}$ by induction, which concludes the proof of the theorem.
\end{proof}

As a corollary, we can give an affirmative answer to the ``if" direction of Conjecture 3.6 from \cite{EKM}.

\begin{proof}[Proof of Corollary \ref{mjmkcor}] We may assume $k$ is odd.  Clearly it is enough to show that $L_jL_k\subset M_{j+k-1}$.  Let $a \in A, x\in L_{j-1}, y\in L_k$.  Then we have $[x,a]y=[x,ay]-a[x,y]$.
The LHS is a completely general generator of $L_jL_k$. By Theorem \ref{mainident}, both the terms on the RHS are in $M_{j+k-1}$.
\end{proof}
\section{Applications of Theorem \ref{mainident} to the modules $B_k$}
\begin{lemma}\label{B4cor} $B_4 = [A_{\leq 2},B_3].$\end{lemma}
\begin{proof}
In \cite{AJ}, it has been shown that
$$B_4 = [A_{\leq 2},B_3]+ \sum_{x,y,z\in A_1} [x[y,z],B_3].$$
Thus our task is to show that each summand $[x[y,z],B_3]$ is actually contained in the span of iterated brackets with only quadratic or linear outermost expression.
We consider a general element of this form, $[x[y,z],u,v,w]$, where we assume that $u$ and $v$ are either linear or quadratic ($w$ is arbitrary).  Then we have:
$$[x[y,z],u,v,w]=[u,x[y,z],v,w]+[[x[y,z],u],v,w].$$
The Leibniz rule implies that $[x[y,z],u]-[x,y][z,u]\in M_3$.  Thus by Theorem \ref{mainident}, we have:
\begin{align}[[x[y,z],u],v,w]&=[[x,y][z,u],[v,w]]=[[z,u],[x,y][v,w]]+[[x,y],[z,u][v,w]]\nonumber\mod L_5\\
&=[[z,u],[x,y[v,w]]] + [[x,y],[z,u[v,w]]] \mod L_5.\label{B4eqn}\end{align}
We now observe that each of the above expressions has only a single term of degree higher than two.  By repeated use of Jacobi identity, we can put that term in the innermost slot.  We have established the corollary.
\end{proof}

\begin{proof}[Proof of Corollary \ref{Bmcor}]
\textbf{Case 1: m is odd.}

Let us consider the element $[x[y,z],w,v] \in B_k$, where $v=[v_1,\ldots,v_{k-2}]$ is some iterated bracket expression.  By Lemma 5.2 of \cite{AJ}, we have:
$$[x[y,z], [w, v]] - [x, [w[y, z], v]] + [y, [w[x, z], v]] - [z ,[w[x, y], v]] \in L_4(x,y,z,w,v),$$
meaning that we regard $v$ as a variable, rather than as an iterated commutator.  Each summand in $L_4$ is necessarily an iterated bracket of linear and quadratic expressions in the variables $x,y,z,w,v$.  By repeated application of the Jacobi identity, we may assume that all summands are of the form $[a_1,\ldots,a_4v]$ for various permutations $\{a_1,\ldots,a_4\} = \{x,y,z,w\}$. We now plug in $v=[v_1,\ldots,v_{k-2}]$. Theorem \ref{mainident} implies $[a_3,a_4v]\in L_{k-1}$, which concludes the proof.

\textbf{Case 2: m is even.}

We have already addressed the case $k=4$. Let us consider the element $[x[y,z],u,v,w]$ of $B_k$, where $w=[w_1,\ldots,w_{k-3}]$.  We may repeat the arguments of Lemma \ref{B4cor} up to equation \eqref{B4eqn} without significant modification.  By Theorem \ref{mainident}, both terms on the RHS are in fact zero in $B_k$, which concludes the proof.
\end{proof}

\begin{proof}[Proof of Corollary \ref{thm:strongerlambdabound}] The proof of \cite{AJ}, Lemma 5.4, applies here without significant modification.\end{proof}

\begin{proof}[Proof of Corollary \ref{cor:lambdabound}]
First, let $n=2$; the first corollary implies that if $\mathcal{F}_\lambda$ occurs in $B_k(A_2)$, then $|\lambda|\leq 2k-3$. Moreover, the modules $\mathcal{F}_{(r)}$ do not appear in $B_k(A_2)$, for $k\geq 2$, as $\CC[z_1]$ is a commutative algebra. Thus $\lambda$ must have at least two rows, so that $|\overline{\lambda}|\leq 2k-5$.

We proceed by induction on $n$. Consider some $\flambda$ which occurs in $B_k(A_n)$. If $\lambda$ has less than $n$ rows, let $\mu = (\lambda_1,\ldots,\lambda_{n-1})$, then it follows that $\mathcal{F}_\mu$ occurs in $B_k(A_{n-1})$, so that $|\overline{\lambda}|=|\overline{\mu}|\leq 2k-5$, by the induction hypothesis.  Thus we may assume $\lambda$ has exactly $n$ rows.  Then the first corollary implies:
\[
|\overline{\lambda}| = |\lambda|-n \leq 2k-3 + 2\left\lfloor \frac{n-2}{2}\right\rfloor -n \leq 2k-5.
\]
\end{proof}

Let $\lambda$ be a Young diagram, and recall that $\mathcal{G}_\lambda(V)=\mathbb{S}_\lambda(V)\ot\Omega^{ev}(V)$. Proposition \ref{glambdabound} now follows, using the techniques of Section 3 and 4.

For a graded vector space $M$, recall that $M[d]$ denotes the $d$th graded component.  Then we have:

\begin{proof}[Proof of Corollary \ref{weakbound}] By Lemma \ref{B4cor} and Theorem \ref{Bmcor}, we have a surjections $A_{\leq 2}^{\ot k-2}\ot\Omega^{ev}\twoheadrightarrow B_k$.  $A_{\leq 2}^{\ot k-2}$ is finite dimensional, while $\dim\Omega^{ev}[d]$ is a polynomial of degree $m+n-1$ in $d$.
\end{proof}

\section{Open questions and conjectures}
While Corollary \ref{weakbound} states that the $B_k(A_{m|n})$ exhibit polynomial growth, computer experiments suggest something stronger:

\begin{conjecture}
For all algebras $A_{m|n}$, the Hilbert series of $B_k$, $k\geq 3$ are rational functions with denominator $\prod_i(1-u_i)\prod(1-v_j)$.
\end{conjecture}

In the completely even case, this is a theorem of \cite{DE}, following from the description of $B_k$ as a finite iterated extension of the modules $\flambda$.  In general, this conjecture would follow as a corollary from the following strengthening of Proposition \ref{glambdabound}:  
\begin{conjecture}\label{glambdaconj}
Let $k\geq 3$. There exists a finite collection $\Lambda$ of Young diagrams $\lambda$, and an isomorphism of abstract Schur functors $V\mapsto B_k(T(V))\cong \oplus_{\lambda\in\Lambda}\mathcal{G}_\lambda$.
\end{conjecture}

One approach to this conjecture is to consider the family $A=A_n$ as $n\to\infty$, along the lines of Sections 3 and 4.  While each $B_k(A_n)$ has a finite length Jordan-H\"older series consisting of modules $\flambda$ with $|\overline{\lambda}|\leq 2k-5$, there is currently no control over the first column of $\lambda$ as $n$ varies.  The conjecture states that as $n$ grows, the new diagrams to appear are precisely those obtained from an earlier diagram by adding two boxes in the first column.  Towards a possible explanation, we observe that not only the Lie algebra $W_n$, but a larger Lie algebra acts on each $B_k$:

\begin{proposition}
Let $A=A_n$.  We have $\textrm{Der}(A/M_3)\cong \widetilde{W_n}:=\oplus_i \Omega_\ast^{ev} \ot\partial_i$, with the Lie bracket:
$$[a\partial_i,b\partial_j]=(a\ast\partial_i(b))\partial_j-(b\ast\partial_j(a))\partial_i,$$
where $\partial_i(a)$ is the Lie derivative, and $\ast$ is the Fedosov product. 
\end{proposition}
\begin{proof}
$A/M_3$ is an $A$-bimodule via the projection $\pi:A\to A/M_3$.  Clearly,
$$\textrm{Der}_A(A,A/M_3)\cong A/M_3\ot V^*\cong \Omega_{\ast}^{ev}(V)\ot V^*,$$ by Theorem \ref{thm:fsisomorphism}.  Since any derivation of $A$ preserves the ideal $M_3$, it descends to a derivation of $\textrm{Der}(A/M_3)$.\end{proof}

\begin{remark} $\widetilde{W_n}$ has a nilpotent ideal consisting of the forms of nonzero degree; the quotient by this ideal may be identified with $W_n$.  In general, for a symmetric tensor category $\cC$, and $V\in\cC$, one can define the analogous Lie algebras $\widetilde{W(V)}$ and $W(V) \in\cC$-Lie.  In particular, for $A=A_{m|n}$, one has the Lie algebras $\widetilde{W_{m|n}}$ and $W_{m|n}$; however the kernel of the map $\widetilde{W_{m|n}}\to W_{m|n}$ can be rather large; for instance in the completely odd setting, $W_{0|n}$ is finite dimensional, while $\widetilde{W_{0|n}}$ is infinite dimensional.
\end{remark}

It follows as in \cite{FS} that the Lie algebra $\widetilde{W_n}$ acts on each $B_k$.  The following conjecture can be checked for $\overline{B}_1, B_2,$ and $B_3$ from their explicit descriptions.  Let us denote the $p$th multi-graded part of a multi-graded vector space $M$ by $M[p]$.  Let $E:=\sum x_i\partial_i$ denote the Euler operator.

\begin{conjecture}\label{neweuler}
The operator $D:=[x_{n+1},x_{n+2}] E\in \widetilde{W_{n+2}}$ is an injection $D:B_k[d_1,\ldots,d_n,0,0]\hookrightarrow B_k[d_1,\ldots,d_n,1,1].$
\end{conjecture}

Let $M$ be a $W_n$-module.  We call $v\in M$ \emph{singular} if $\partial_i v=0$ for  $i=1,\ldots n$, and we denote by $M^{sing}$ the vector subspace of singular vectors.  $M^{sing}$ is a $\gl_n$-submodule of $M$.  Let $\mathcal{O}$ denote the category of $W_n$-modules satisfying the conditions of Theorem \ref{thm:rudakovwn}, and let $\mathcal{O}^\circ$ denote the full subcategory of $\mathcal{O}$ consisting of modules whose Jordan-H\"older series do not contain $\lambda=(1^k)$ for any $k$.  The following proposition is a consequence of Theorem \ref{flambdairrep}, together with the character formula for $\flambda$:

\begin{proposition}\label{singprop} Every $M\in\mathcal{O}^\circ$ is generated by $M^{sing}$ as a $W_n^0$-module.  The multiplicity of $\flambda$ in the Jordan-H\"older series of $M$ is equal to the multiplicity of $V_\lambda$ in $M^{sing}$.
\end{proposition}

\begin{proposition}[see, e.g, \cite{AJ}, Lemma 5.1] For $k\geq 3$, and for all $n$, \mbox{$B_k(A_n)\in\mathcal{O}^\circ$}.\end{proposition}

The operator $D$ preserves the vector subspace $B_k^{sing}$, and moreover maps highest weight vectors to highest weight vectors for the $\gl_n$ action.  Let $m_{k,\lambda}$ denote the multiplicity of $V_\lambda$ in $B_k$ (and thus by Proposition \ref{singprop}, the multiplicity of $\flambda$ in $B_k$).  Then \ref{neweuler} would imply that $m_{k,\lambda} \leq m_{k,(\lambda,1,1)}$.  On the other hand, Conjecture \ref{glambdabound} implies that the $m_{k,\lambda}$ are bounded from above, uniformly in $n$ and $\lambda$, since each $\mathcal{G}_\mu$ contains any $\flambda$ (regarded here as a $\gl_n$-module) as a summand at most once.  Thus Conjecture \ref{neweuler} would imply that the sequence $m_{k,(\lambda,1^{2j})}$ is both non-decreasing, and bounded from above, hence eventually constant, implying Conjecture \ref{glambdaconj}.


\end{document}